\documentclass[12pt,a4paper]{article}

\usepackage[margin=1in]{geometry}
\usepackage[T1]{fontenc}
\usepackage{microtype}

\usepackage{amsmath,amssymb,amsthm}
\IfFileExists{newtxtext.sty}{\usepackage{newtxtext,newtxmath}}{\usepackage{mathptmx}}
\numberwithin{equation}{section}

\usepackage{array}
\usepackage{tabularx}
\usepackage{booktabs}
\usepackage{longtable,colortbl,float}
\usepackage{enumitem}
\usepackage{needspace}\usepackage{tikz}\usetikzlibrary{arrows.meta,positioning} 

\newcommand{\corestep}[2]{%
  \par\addvspace{0.9\baselineskip}\noindent
  \begin{tikzpicture}
    \node[draw=gray!50, line width=0.4pt, rounded corners=2pt,
      top color=gray!12, bottom color=gray!9,
      inner xsep=10pt, inner ysep=8pt,
      text width=\dimexpr\linewidth-21pt\relax, align=left]
      {\textbf{#1}\par\smallskip #2};
  \end{tikzpicture}\par\nobreak\smallskip
}

\newcommand{\softbox}[1]{\begingroup\setlength{\fboxsep}{7pt}%
\fcolorbox{gray!50}{gray!11}{$\displaystyle #1$}\endgroup}

\newcommand{\quietemphasis}[1]{\par\addvspace{5pt}\noindent\emph{#1}\par\addvspace{5pt}}

\newcommand{\plainstep}[2]{\par\addvspace{0.8\baselineskip}\noindent
\textbf{#1}\par\nobreak\smallskip\noindent #2\par\addvspace{0.6\baselineskip}}
\newcommand{\plainbox}[1]{\begingroup\displaystyle #1\endgroup}
\newenvironment{plainstatement}{\par\addvspace{7pt}\noindent
\begin{minipage}{\linewidth}\raggedright}
{\end{minipage}\par\addvspace{7pt}}

\theoremstyle{plain}
\newtheorem{hubtheorem}{Theorem}[section]

\newtheorem{hublemma}[hubtheorem]{Lemma}

\theoremstyle{definition}

\theoremstyle{remark}

\usepackage[hidelinks,hypertexnames=false]{hyperref}

\hypersetup{pdfauthor={Gabriel Ben-Simon},
    pdftitle={
        Cauchy-Type Convolution Integrals as the Starting Point for a First Encounter with Complex Analysis
    }
}

\title{
    \textbf{Cauchy-Type Convolution Integrals as the Starting Point for a First Encounter with Complex Analysis}
}
\author{
    Gabriel Ben-Simon\\
    \small Afeka Academic College of Engineering\\
    \small Tel Aviv, Israel\\
    \small\texttt{GavrielBe@afeka.ac.il}
}
\date{}

\begin{document}
\maketitle
\begin{abstract}
Alongside three established presentations of elementary complex analysis, we propose an approach centered on one organizing question: when is contour convolution well defined on homotopy classes? From this starting point, we develop the foundational results of the subject. A supporting conceptual distinction identifies those elements of real and complex analysis that share the same operational and symbolic rules. We refer to this collection as the ``Toolbox.'' Examples and structural connections illustrate the explanatory value and reach of the approach across different aspects of the theory. We also examine the relationship of our approach to the established presentations.
\end{abstract}

\section{Introduction}
\label{sec:introduction}

There is an open secret in the classical organization of a first course in
complex analysis: the theory begins in earnest with the Cauchy--Goursat theorem.
Three familiar approaches, which may overlap, provide a useful context:
\begin{enumerate}[label=(\roman*),leftmargin=2.3em,itemsep=3pt,topsep=5pt]
\item \emph{Power series as a starting point}: for example, Cartan and
Lang~\cite{Cartan1963,Lang1999}.
\item \emph{The geometric point of view}: for example,
Needham~\cite{Needham2023}.
\item \emph{The classical course development}: from complex numbers and
differentiability to the Cauchy--Goursat theorem and its consequences;
see, for example, Ahlfors, Bak--Newman, Saff--Snider, and Isaev~\cite{Ahlfors1979,BakNewman2010,SaffSnider2003,Isaev2017}.
\end{enumerate}
The approach proposed here takes a different starting point.
\quietemphasis{The central idea of this paper can be expressed through a single question and a single answer:}
\begin{center}
\begin{tikzpicture}
\node[draw=gray!55,line width=0.55pt,rounded corners=3pt,
 fill=white,inner sep=11pt,
 text width=\dimexpr\linewidth-24pt\relax,align=center] {
\textit{Q: What is the core of the basics of complex analysis?}\\[7pt]
A: The study of when convolution is well defined on homotopy classes of paths.
};
\end{tikzpicture}
\end{center}
More precisely, given a domain \(D\subset\mathbb C\) and a vector space
\(V\subset C(D)\), we ask when the operation
\begin{equation}\label{eq:s1-seed}
 (f*_\gamma g)(z)=\int_\gamma f(z-w)g(w)\,dw
\end{equation}
depends only on the homotopy class of the path.
Both factors must be defined throughout the path and its deformation;
Section~2 makes this admissibility requirement precise. The question asks
when \(*_\gamma\) becomes a well-defined operation on \([\gamma]\).
Placing this question and answer at the beginning both reorganizes the
topics of the theory and clarifies their relation to other presentations.

The first step rests on a classical equivalence: the Cauchy--Goursat
theorem gives homotopy invariance of the Cauchy integral away from its
pole. We take the reverse starting point: we assume that the Cauchy
integral is invariant under homotopy. We work here with smooth homotopies; Section~2 also treats the continuous version.
These classical connections are clearly discussed, for example, in
\cite[Chs.~8--10]{StewartTall2018}, \cite[\S8.2.1]{Remmert1991};
for the homotopy formulation, see also \cite[Lectures~9--10]{Isaev2017}.

The purpose of this paper is to place the question in
\eqref{eq:s1-seed} at the center of the development and to show what follows
from keeping it there.
\quietemphasis{Once the operation is defined on homotopy classes,
it provides a common starting point for reconstruction, coefficient
projections, changes of variables, and perturbations. Furthermore, we use these
connections to reorganize the elementary theory, explain its relation to
other presentations, relate it to other mathematical theories, and reach
selected more advanced results.} We also hope to simplify the core
arguments and make their structure more transparent.
The contribution lies in this sustained development of the initial
question, together with a precise separation of the tools used to carry
it out.

\medskip
\noindent\textbf{``The Real Conceptual Toolbox of Complex Analysis.''}
Our core idea is complemented by an important auxiliary concept: the
``Real Conceptual Toolbox of Complex Analysis,'' or simply the
\emph{Toolbox}; ``real'' here refers to real-variable analysis.
The reader is invited to pause here and consult the Toolbox tables in
Appendix~\ref{app:toolbox} to obtain a precise symbolic description of what we call the
Toolbox. Much of the elementary calculus retains the same formal organization when
\(x\) is replaced by \(z\). Limits and continuity obey the familiar rules;
the arithmetic of derivatives and the chain rule retain their form;
parametrized integrals admit reparametrization, integration by parts, and
the fundamental theorem whenever a primitive exists. Choosing a parameter
to simplify the calculation of a line integral, and using vanishing
integrals on closed paths to identify the existence of a potential or
primitive, are familiar tasks from technical exercises.
\quietemphasis{The lesson is clear: the Toolbox embodies the idea that the symbolic
calculus is essentially the same, both in its computational rules and in
its formal structure.}
We will see that this parallelism is inherent in the implementation of
the core ideas. Wherever possible, we will explicitly highlight the
indispensable role of the Toolbox in that implementation. Much of the
development below alternates between the structure supplied by homotopy invariance of the convolution integrals
and the inherited operations that make it effective.\\

To motivate our starting point, consider the role of the Toolbox.
It is an indispensable part of complex analysis. Viewed solely through
its scalar correspondences, however, it does not yet reveal the analytic
structure that distinguishes the theory. This inherited apparatus defines
complex differentiability; the definition alone does not explain the
consequences that make complex analysis a theory of its own.

\medskip
\noindent\textbf{The primitive theorem.}
A possible motivational starting point is the following. Looking at the
tables comparing real and complex scalar calculus, we recover the
Newton--Leibniz formula under suitable conditions and construct a
differentiable function as an antiderivative. As suggested above, this
formal parallel alone does not explain the specifically complex theory.
Let us now turn to the corresponding line-integral viewpoint and consider
integrals along closed paths. What does the primitive theorem tell us?

\begin{plainstatement}
Let \(D\) now be simply connected and \(f\in C(D)\). For piecewise
\(C^1\) paths in \(D\), the following are equivalent:
\begin{enumerate}[label=(\roman*),leftmargin=2.3em,itemsep=2pt,topsep=5pt]
 \item \(f\) has a complex primitive \(F\) on \(D\).
 \item The integral of \(f\) around every closed path is zero.
 \item Integrals of \(f\) between fixed endpoints are independent of the path.
 \item Integrals of \(f\) are invariant under homotopies with fixed endpoints.
\end{enumerate}
The first three conditions are equivalent on any domain; simple
connectivity makes the fourth equivalent to them.
\end{plainstatement}
\quietemphasis{Condition (iv) marks the crucial distinction here: homotopy invariance of
the complex integral forces analytic structure on its continuous
integrand.} Although all four conditions can also be expressed through
real line integrals, the complex multiplication ties the component
integrals together. Thus, the line-integral viewpoint reveals a rigidity that goes
beyond the symbolic calculus recorded in the Toolbox.

In convolution notation, the closed-path condition is
\(1*_{[\gamma]}f=0\), where \(1\) denotes the constant kernel and
\(\gamma\) is closed. Before descent is justified, one writes
\(1*_\gamma f\).
To state explicitly what this descent to homotopy classes entails:
a continuous complex function satisfying the closed-path condition is
itself analytic, as the Morera argument will show.

The condition does much more than produce one differentiable function:
it forces analytic structure on the original integrand.

The inherited tools provide the language and the point of entry; the
invariance condition reveals the analytic rigidity. It is in this sense
that the Toolbox, considered as the theory itself, is an ``empty shell'':
it helps us reach the core but does not by itself explain it.

\quietemphasis{We now pursue a fundamental mathematical strategy: understanding a phenomenon by placing it in a more general framework.}
We are thus led to our main question: given a domain
\(D\subset\mathbb C\) and a vector space \(V\subset C(D)\), when can
\(f*_{[\gamma]}g\) be defined for functions in \(V\)? We aim to show the
advantages of placing this question at the beginning and retaining it as
the central question of the development. Together with the distinction and interplay
between the Toolbox and the structure it serves, this is the organizing
principle of the proposed presentation of elementary complex analysis.

The remainder of the paper is organized as follows. Section~2 shows how
our guiding principle generates the foundational theorems on which the
theory is built. Although the mathematics is elementary, we retain the
core proofs because their arguments are central to the paper and to the
development of its guiding idea. It also yields naturally an
\emph{analytic resolution of the identity}---a term borrowed to describe the reconstruction of a function
from its power-series components. We thus obtain both the foundations
and the theory's ``computational arm'', and briefly indicate how this
starting point relates to other approaches to a first course.

Section~3 explores the reach of the approach through a sampling of arguments
and results, including structural parallels with other fundamental
theories. Further results of elementary complex analysis can be treated along similar lines.
Its principal themes are how the core idea yields elegance of
argument, economy of symbolic calculation, and connections with more
advanced topics, while, where appropriate, maintaining a close conceptual
connection between the results and convolution on homotopy classes.

Section~4 draws together the main conclusions and discusses possible
implications for teaching. Appendix~\ref{app:toolbox} gives the detailed
symbolic comparisons constituting the Toolbox.

\medskip
This paper belongs to a group of works arising from a perspective that
has crystallized in recent years, following two decades of teaching
complex analysis, ordinary and partial differential equations, Fourier
analysis, and linear algebra. These works propose that substantial parts
of such courses can be organized around a small number of conceptual
cores, offering a clearer understanding together with computational,
technical, explanatory, and pedagogical consequences, and possible
relevance to AI-assisted mathematical work. Related developments concern
ordinary differential equations~\cite{BenSimonODE2026} and partial
differential equations~\cite{BenSimonPDE2026}, the latter in preparation.

\section{From Homotopy Invariance to the Foundational Results of Complex Analysis}
\label{sec:cauchy-convolutional-hub}

This section develops the foundational theorems of complex analysis from
homotopy invariance of convolution integrals. Starting with continuous
functions, we obtain reproduction, analyticity, the resolution of the
identity---here, reconstruction from power-series components---and the derivative formulas; the same mechanism yields Morera's
theorem and the Cauchy--Goursat vanishing conclusion. The arguments are
elementary; we retain their essential details because their sequence is
central to the approach.

Throughout, \emph{holomorphic} means complex differentiable on an open set,
whereas \emph{analytic} means locally representable by a convergent power
series. Their equivalence is established in the course of the construction.

\subsection{The operation and its domains}

Let \(D\subset\mathbb C\) be a domain and \(V\subset C(D)\) a vector
space. We examine the descent question for each \(f\in V\).
For a continuous kernel \(K\) on an open set \(U\), define
\begin{equation}\label{eq:s2-convolution}
 (K*_\gamma f)(z)=\int_\gamma K(z-\zeta)f(\zeta)\,d\zeta,
 \qquad \gamma\subset D\cap(z-U).
\end{equation}
Paths are piecewise \(C^1\). Homotopies are continuous, with endpoints fixed
for open paths; for closed paths we allow free homotopies through loops.
For fixed \(z\), the admissible domain is \(D\cap(z-U)\).
The notation \(K*_{[\gamma]}f\) is justified precisely when the value is the
same for any two admissible paths in that homotopy class.

For \(K=1\) on a simply connected domain \(D\), the descent question for
\(1*_\gamma f\) is precisely the primitive criterion recalled in the
Introduction.

The kernel \(\kappa_k(w)=w^{-k}\) is defined on
\(\mathbb C\setminus\{0\}\) for \(k\geq1\); put \(\kappa_0=1\).
We normalize its convolution in the customary Cauchy orientation:
\begin{equation}\label{eq:s2-operators}
 C_{\gamma,k}f(z)
 :=\frac{(-1)^k}{2\pi i}(\kappa_k*_\gamma f)(z)
 =\frac1{2\pi i}\int_\gamma\frac{f(\zeta)}{(\zeta-z)^k}\,d\zeta,
 \qquad k=0,1,2,\ldots .
\end{equation}
For \(k=0\) there is no excluded point and the value is independent of \(z\).
For \(k\geq1\), paths and homotopies must avoid \(z\).
Write \(C_\gamma=C_{\gamma,1}\).
The factor \((-1)^k\) records the sign difference between \(z-\zeta\) in
convolution and \(\zeta-z\) in the standard Cauchy kernel.

\corestep{From invariance to reproduction.}{%
Homotopy invariance of the Cauchy integral yields the Cauchy reproduction
formula for continuous functions.}

The structural hypothesis on a continuous function is
\begin{equation}\tag{H}\label{eq:s2-H}
 \softbox{\begin{gathered}
 \textbf{Condition (H)}\\[3pt]\text{For every }z\in D,\quad C_\gamma f(z)\text{ depends only on }[\gamma]\\
 \text{for closed paths in }D\setminus\{z\}.
 \end{gathered}}
\end{equation}
Here \(\gamma\) is closed. The singularity of the kernel is essential:
a positive circle about \(z\) cannot contract in the admissible domain.

We record the precise statement and proof of this classical implication.

\begin{hubtheorem}[Reproduction]\label{thm:s2-reproduction}
Suppose \(f\in C(D)\) satisfies \eqref{eq:s2-H}.
If \(\gamma\) is homotopic in \(D\setminus\{z\}\) to a small positively
oriented circle about \(z\), then \(C_\gamma f(z)=f(z)\).
In particular, for every closed disk \(\overline{B(a,R)}\subset D\),
\begin{equation}\label{eq:s2-reproduction}
 C_{\partial B(a,R)}f=f\quad\text{on }B(a,R),
\end{equation}
with positive boundary orientation.
\end{hubtheorem}

\begin{proof}
For sufficiently small \(r>0\), let \(c_r(t)=z+re^{it}\).
Invariance and the parametrization of the circle give
\[
 C_\gamma f(z)=C_{c_r}f(z)
 =\frac1{2\pi}\int_0^{2\pi} f(z+re^{it})\,dt.
\]
Thus
\[
 |C_\gamma f(z)-f(z)|
 \leq\max_{|\zeta-z|=r}|f(\zeta)-f(z)|\longrightarrow0.
\]
The left side is independent of \(r\), so it vanishes.

For completeness, when \(z\in B(a,R)\), choose
\(0<r<R-|z-a|\). The circles
\[
 H(s,t)=(1-s)a+sz+\big((1-s)R+sr\big)e^{it}
\]
join \(\partial B(a,R)\) to \(c_r\). They lie in
\(\overline{B(a,R)}\) and avoid \(z\), since their radius exceeds
\((1-s)|a-z|\). The first part therefore applies to every interior \(z\).
\end{proof}

The role of the Toolbox is transparent in this argument.
The contour condition is a homotopy condition in the punctured domain;
on a multiply connected domain it cannot be replaced merely by
``winding number one about \(z\).''

\corestep{From reproduction to analytic reconstruction.}{%
Cauchy reproduction yields analyticity, the power-series resolution of the
identity, and the higher derivative formulas.}

We record the precise statement and proof of this classical implication.

\begin{hubtheorem}[Analytic reconstruction]\label{thm:s2-analytic}
Let \(f\) be continuous on \(\overline{B(a,R)}\), and suppose
\(C_\gamma f=f\) on \(B(a,R)\), where
\(\gamma=\partial B(a,R)\) is positively oriented. Then
\begin{equation}\label{eq:s2-taylor}
 f(z)=\sum_{n=0}^{\infty}a_n(z-a)^n,
 \qquad a_n=C_{\gamma,n+1}f(a),
 \qquad |z-a|<R.
\end{equation}
The series converges locally uniformly. The function has derivatives of
all orders, and throughout the disk
\begin{equation}\label{eq:s2-derivatives}
 C_{\gamma,k}f=\frac{1}{(k-1)!}\,\partial^{\,k-1}f,
 \qquad k\geq1,\qquad \partial f=f'.
\end{equation}
\end{hubtheorem}

\begin{proof}
For \(|z-a|\leq r<R\) and \(\zeta\in\gamma\), expand
\[
 \frac1{\zeta-z}=\sum_{n=0}^{\infty}
 \frac{(z-a)^n}{(\zeta-a)^{n+1}}.
\]
The remainder after degree \(N\) is bounded uniformly by
\(R^{-1}(r/R)^{N+1}/(1-r/R)\). Multiplication by the bounded boundary
values of \(f\) and the contour estimate justify termwise integration.
Reproduction gives \eqref{eq:s2-taylor}. If
\(M=\max_\gamma|f|\), the same estimate gives \(|a_n|\leq MR^{-n}\).
For each fixed derivative order, the differentiated series is therefore
uniformly convergent on every smaller disk: the additional polynomial
factors in \(n\) are dominated by the geometric decay \((r/R)^n\).

Finally, on any compact subset of the disk the distance to \(\gamma\)
is positive. The kernel and each of its derivatives are uniformly bounded
there, so differentiation under the integral is legitimate and gives
\(\partial C_{\gamma,k}f=kC_{\gamma,k+1}f\).
Starting from \(C_\gamma f=f\) proves
\eqref{eq:s2-derivatives} at every interior point.
\end{proof}

The geometric-series expansion makes particularly clear the symbiotic
role of the Toolbox in this argument. The power-series expansion obtained
here is the first manifestation of the resolution of the identity.

For a positive disk boundary the operator dictionary is
\begin{equation}\label{eq:s2-dictionary}
 \plainbox{\quad
 C_{\gamma,0}=0,\qquad
 C_{\gamma,1}=I,\qquad
 C_{\gamma,k}=\frac{\partial^{\,k-1}}{(k-1)!}\ (k\geq2).
 \quad}
\end{equation}
For the first equality, termwise integration gives a primitive in the
open disk, so the integral is zero on every concentric circle of radius
\(r<R\). Continuity on the closed disk gives the same conclusion on its
boundary by letting \(r\) tend to \(R\). The global contour qualifications are
established below. The identities act on functions satisfying the reconstruction
hypothesis, with values in the disk.

At this point the same starting principle reaches further into the
foundations: it also yields Morera's theorem, returning us to the
vanishing-integral condition that motivated our initial discussion of primitives.

\corestep{Morera through the same reconstruction mechanism.}{%
Vanishing integrals on triangle boundaries give local primitives.
For these primitives, elementary differentiation establishes the deformation
invariance needed for reproduction. The preceding reconstruction argument
then forces the original continuous function to be analytic.}

We give the precise statement and proof.

\begin{hubtheorem}[Morera]\label{thm:s2-morera}
Let \(f\in C(D)\). Suppose the integral of \(f\) vanishes on the oriented
boundary of every triangle whose closed interior lies in \(D\).
Then \(f\) is analytic on \(D\). It suffices that this hypothesis hold
in a neighborhood of each point.
\end{hubtheorem}

\begin{proof}
Work in a convex disk \(U\subset D\), fix \(b\in U\), and define
\(F(z)=\int_{[b,z]}f(\zeta)\,d\zeta\).
The triangle with vertices \(b,z,z+h\) gives
\[
 F(z+h)-F(z)=\int_{[z,z+h]}f(\zeta)\,d\zeta.
\]
By continuity,
\[
 \frac{F(z+h)-F(z)}h=\int_0^1 f(z+th)\,dt\longrightarrow f(z),
\]
so \(F'=f\) on \(U\).
Since \(f\) is continuous, \(F\) is real \(C^1\):
\(F_x=f\) and \(F_y=if\).

We next prove the deformation invariance needed for reproduction of \(F\),
without appealing to Goursat. Fix \(z\in U\) and put
\(q(\zeta)=F(\zeta)/(\zeta-z)\).
This function is complex differentiable with continuous derivative on
\(U\setminus\{z\}\). For a smooth homotopy \(H(s,t)\) of closed paths
there, differentiation under the integral gives
\[
 \begin{aligned}
 \frac{d}{ds}\int_0^{2\pi}q(H)H_t\,dt
 &=\int_0^{2\pi}\big(q'(H)H_sH_t+q(H)H_{st}\big)\,dt\\
 &=\big[q(H)H_s\big]_{t=0}^{t=2\pi}=0.
 \end{aligned}
\]
The endpoint terms cancel because the paths are closed.
All quantities involved are continuous on the compact parameter rectangle,
and its image avoids \(z\), which justifies the differentiation.

Apply this calculation to the explicit \(C^\infty\) circle homotopies in the proof of
Theorem~\ref{thm:s2-reproduction}, with
\(\overline{B(a,R)}\subset U\) and any \(z\in B(a,R)\).
The shrinking-circle estimate for the continuous function \(F\) gives
\(C_{\partial B(a,R)}F=F\) throughout the disk.
Theorem~\ref{thm:s2-analytic} supplies a power series for \(F\).
Its termwise derivative is a power series for \(f=F'\).
Such disks cover \(D\), proving the claim.
\end{proof}

Morera starts with vanishing \(1\)-kernel periods; the preceding theorems
start with homotopy invariance for the \(1/z\)-kernel.
The proof connects these hypotheses through a primitive and then uses the
same reproduction mechanism.

\corestep{From invariance to Cauchy--Goursat.}{%
Condition (H) yields the Cauchy--Goursat vanishing conclusion through
analyticity and local primitives.}

We first record explicitly how local primitives control arbitrary
continuous homotopies.

\begin{hublemma}\label{lem:s2-local-primitives}
If a continuous function \(h\) has a primitive in a neighborhood of every
point of an open set \(\Omega\), its integrals are invariant under homotopies
of paths with fixed endpoints and under free homotopies of closed paths
in \(\Omega\). In particular, its integral on a null-homotopic loop is zero.
\end{hublemma}

\begin{proof}
Cover the compact image of a homotopy by disks on each of which a primitive
exists. Subdivide its parameter square into a sufficiently fine rectangular
grid so that the image of each cell lies in one such disk; this is possible
by the Lebesgue number lemma. Join images of neighboring vertices by straight
segments. A shared segment belongs to both cell disks, since disks and their
intersections are convex. The integral around each resulting cell is zero
by its local primitive. Summing cancels the interior edges.
On the exterior edges, replace each segment by the corresponding piece of
the original endpoint path: both lie in the same disk and have the same
primitive difference. For fixed-endpoint homotopies the other two edges
are constant. For free homotopies of loops their polygonal replacements
coincide with opposite orientations. Thus the two path integrals agree.
A constant loop has integral zero.
\end{proof}

Under \eqref{eq:s2-H}, Theorems~\ref{thm:s2-reproduction}
and~\ref{thm:s2-analytic} give local power series. Integrating such a series
termwise gives a local primitive. Lemma~\ref{lem:s2-local-primitives}
therefore yields the Cauchy--Goursat conclusion:
\begin{equation}\label{eq:s2-zero-periods}
 C_{\gamma,0}f=0
 \quad\text{for every closed }\gamma\text{ null-homotopic in }D.
\end{equation}
On a simply connected domain this holds for every closed path and gives a
global primitive. No such global conclusion is asserted on a domain with
holes. Likewise, if \(f\) is analytic, then for fixed \(z\) each weighted
integrand in \eqref{eq:s2-operators} is analytic on \(D\setminus\{z\}\)
and has local primitives. Thus all the \(C_{\gamma,k}\),
\(k\geq1\), descend to homotopy classes there.

Conversely, the classical Goursat theorem derives homotopy invariance
from complex differentiability alone, without continuity of the derivative;
see, for example, \cite[Ch.~4]{Ahlfors1979}.

Consequently, for \(f\in C(D)\), the following characterizations agree:
\[
 \softbox{\begin{gathered}
 \text{condition (H)}\ \Longleftrightarrow\ \text{local Cauchy reproduction}\\
 \Longleftrightarrow\ \text{analyticity}\ \Longleftrightarrow\ \text{holomorphy}.
 \end{gathered}}
\]
Here local reproduction means reproduction on every disk compactly
contained in \(D\). The integral-first development establishes the first
three characterizations directly; Goursat supplies the implication from
complex differentiability alone.

\quietemphasis{This equivalence is a meeting point at the foundations,
not merely a correspondence between later consequences.}
\label{par:geometric-meeting}
Cauchy's Green-theorem argument and Goursat's subdivision proof make
geometry operative at the stage where integral invariance is established,
under different regularity assumptions. Our route instead assumes
invariance for a continuous density and develops the analytic structure
it forces. The two starting points therefore reach the same foundational
junction, from which both component calculus and geometric explanation
can develop. Appendix~\ref{app:foundational-map} displays the forward
construction and records how the classical routes enter it.

\medskip
\noindent\textbf{Notation for the resolution of the identity.}
We collect the preceding construction in compact operator notation.
For analytic \(f\) on a neighborhood of \(\overline{B(a,R)}\), let
\(\gamma=\partial B(a,R)\) be positive and write
\(C_{[\gamma],k}=C_{\gamma,k}\). The coefficients below are independent
of the radius as long as the contour deformation remains admissible.
\begin{plainstatement}
\begin{equation}\label{eq:s2-projections}
 e_n^{(a)}(z)=(z-a)^n,\qquad
 P_n^{(a)}f=\big(C_{[\gamma],n+1}f(a)\big)e_n^{(a)},\quad n\geq0.
\end{equation}
\begin{equation}\label{eq:s2-resolution}
 I=\sum_{n\geq0}P_n^{(a)},\qquad
 P_n^{(a)}P_m^{(a)}=\delta_{nm}P_n^{(a)}.
\end{equation}
The sum converges uniformly on compact subsets of \(B(a,R)\).
Parametrization gives
\(P_n^{(a)}e_m^{(a)}=\delta_{nm}e_m^{(a)}\): the powers are simultaneous
eigenfunctions of these projections. In contrast, differentiation lowers degree:
\[
 C_{\gamma,k}e_n^{(a)}=\binom{n}{k-1}e_{n-k+1}^{(a)}
 \quad(k\geq1,\ n\geq k-1),
\]
and is zero when \(n<k-1\).
\end{plainstatement}
These are not new operators: the notation keeps their common origin and
homotopy dependence visible in the developments of Section~3.

\plainstep{What the construction achieves.}{%
Starting from continuous functions, condition (H) gives Cauchy
reproduction, analyticity, derivatives of every order, the derivative
formulas, and the power-series resolution of the identity. Local primitives
connect this reconstruction mechanism to Morera and to the
Cauchy--Goursat vanishing conclusion. The classical converse identifies
the meeting point with the geometric route. Compact projection notation
records the resulting ``computational arm'' for Section~3.
Appendix~\ref{app:foundational-map} summarizes these dependencies.}

\section{The reach of the starting principle}
\label{sec:reach}

This section develops the explanatory case for our approach.
We illustrate the scope of condition (H) through a cross-section of
applications, interpretations, and structural connections; further
examples can be developed along similar lines. We introduce
notation that condenses and reorganizes the foundational results while
providing a language for further extensions. We first extend the resolution
of the identity to annuli, then examine how our starting principle
illuminates selected results and applications and connects them to more
advanced ideas. We also formulate precise structural parallels with other
theories. Finally, we return to established approaches to a first course in
complex analysis, identifying the key points at which they meet ours.
These connections clarify the place of the proposed approach alongside
familiar presentations and the distinct insights each brings to the subject.

\subsection{From a local resolution to annular projections}

To extend the resolution of the identity, it is illuminating to consider
first an isolated singularity and then a general annulus. We focus on the
general case: let \(A=\{r<|z-a|<R\}\), and suppose \(f\in C(A)\)
satisfies (H).
For \(r<s<|z-a|<t<R\), write \(\gamma_s,\gamma_t\) for the positive
circles about \(a\). Cutting the region between them and excising a small
circle \(c_z\) about \(z\) gives
\[
 C_{[\gamma_t]}f(z)-C_{[\gamma_s]}f(z)=C_{[c_z]}f(z)=f(z).
\]
The internal edges cancel by the local primitive argument of Section~2.
The contour representatives may vary within their admissible classes;
the relation between the three contours is the boundary relation in
Figure~\ref{fig:s3-annulus}.

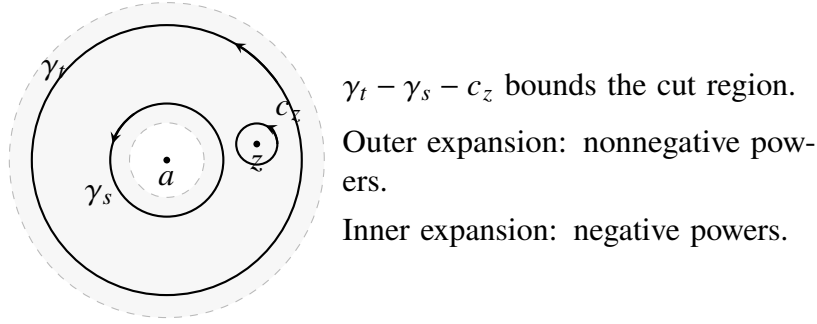
\begin{figure}[htbp]
\centering
\begin{tikzpicture}[scale=0.85,>=stealth]
\fill[gray!6] (0,0) circle (2.45);
\fill[white] (0,0) circle (0.58);
\draw[dashed,gray!55] (0,0) circle (2.45);
\draw[dashed,gray!55] (0,0) circle (0.58);
\draw[thick] (0,0) circle (2.1);
\draw[->,thick] (35:2.1) arc (35:60:2.1);
\draw[thick] (0,0) circle (0.88);
\draw[->,thick] (125:0.88) arc (125:163:0.88);
\fill (0,0) circle (1.5pt) node[below] {$a$};
\fill (1.4,0.25) circle (1.5pt) node[below] {$z$};
\draw[thick] (1.4,0.25) circle (0.32);
\draw[->,thick] (1.72,0.25) arc (0:65:0.32);
\node at (-1.75,1.4) {$\gamma_t$};
\node at (-1.05,-0.55) {$\gamma_s$};
\node at (1.9,0.75) {$c_z$};
\node[align=left,text width=6.4cm] at (6.5,0) {
\(\gamma_t-\gamma_s-c_z\) bounds the cut region.\\[7pt]
Outer expansion: nonnegative powers.\\[4pt]
Inner expansion: negative powers.};
\end{tikzpicture}
\caption{Three positive contours for annular reconstruction. Their boundary
relation, together with invariance, gives the difference of the two measurements.}
\label{fig:s3-annulus}
\end{figure}

The Toolbox now supplies the same geometric expansion as before, in
\((z-a)/(\zeta-a)\) on the outer contour and in its reciprocal on the
inner contour. Define \(C_{[\gamma],k}\) by the weighted formula
\eqref{eq:s2-operators} for every integer \(k\); for \(k\leq0\) its kernel
is polynomial. This extends the projection notation to
\begin{equation}\label{eq:s3-laurent}
\plainbox{\begin{gathered}
 \lambda_n^{(a)}(f)=C_{[\gamma_\rho],n+1}f(a),\qquad
 P_n^{(a)}f=\lambda_n^{(a)}(f)(z-a)^n,\quad n\in\mathbb Z,\\
 I=\sum_{n\in\mathbb Z}P_n^{(a)},\qquad
 P_n^{(a)}P_m^{(a)}=\delta_{nm}P_n^{(a)}.
\end{gathered}}
\end{equation}
Here \(r<\rho<R\), and reconstruction is locally uniform on \(A\).
Termwise differentiation and integration are again Toolbox operations;
a primitive on the annulus exists exactly when \(\lambda_{-1}^{(a)}(f)=0\),
the sole term that cannot be integrated there by an integer power.

\plainstep{Isolated singularities in projection notation.}{%
For \(0<|z-a|<R\), the residue is the coefficient
\(\operatorname{Res}_a f=\lambda_{-1}^{(a)}(f)\), so
\(P_{-1}^{(a)}f(z)=\lambda_{-1}^{(a)}(f)/(z-a)\).
\par\smallskip\noindent
\emph{Removable:} all negative projections vanish.\par\smallskip\noindent
\emph{Pole of order \(m\):} \(P_{-m}^{(a)}f\ne0\) and
\(P_n^{(a)}f=0\) for \(n<-m\).\par\smallskip\noindent\emph{Essential:} infinitely many
negative projections are nonzero.}
These are the familiar classifications. Their operator notation keeps
reconstruction and its components available for the connections below.

\subsection{Applications: economy, interpretation, and extension}

\subsubsection{Classical results: from economy to further connections}
The examples have progressively different explanatory roles. Liouville,
the identity theorem, and the maximum principle illustrate the economy
of the resolved identity. The argument principle and Rouch\'e reveal a
shared mechanism through change of variable and admissible perturbation.
Cauchy--Pompeiu goes further: the defect of invariance connects contour
deformation, conjugation, and the distinction between holomorphic and
smooth functions. We recall only the arguments needed to expose these roles.

\smallskip\noindent\textbf{\emph{Liouville.}}
For an entire function bounded by \(M\), the estimate already obtained in
Section~2 gives \(|\lambda_n^{(0)}(f)|\leq MR^{-n}\) for every \(R>0\).
Thus all positive components vanish. The same argument makes an entire
function of growth \(O((1+|z|)^m)\) a polynomial of degree at most \(m\).

\smallskip\noindent\textbf{\emph{The identity theorem.}}
The first nonzero term of a local resolution factors out a power and
leaves a nonvanishing factor. Zeros are therefore isolated unless the
function vanishes locally; connectedness propagates that local vanishing.

\smallskip\noindent\textbf{\emph{The maximum principle.}}
Termwise differentiation gives smoothness and the Cauchy--Riemann
equations, hence harmonic real and imaginary parts. At a local maximum
\(a\) of \(|f|\), rotate so that \(e^{-i\theta}f(a)=|f(a)|\).
Then \(u=\operatorname{Re}(e^{-i\theta}f)\leq|f|\) has a local maximum.
The real harmonic maximum principle makes \(u\) locally constant;
the Cauchy--Riemann equations do the same for its conjugate component,
and the identity theorem gives global constancy on the domain.
\\
\textit{These opening examples claim no new theorem or alternative proof:
they exhibit the short, transparent route from (H), through the resolved
identity, to classical arguments. The maximum-principle argument also
uses the real harmonic maximum principle.} 
\newpage
We now turn to more structural uses.

\Needspace{9\baselineskip}

\medskip\noindent\textbf{The argument principle: change of variable.}
Let \(\Gamma\) positively bound a bounded Jordan domain \(U\), with
piecewise \(C^1\) boundary, and let \(F\) be meromorphic near
\(\overline U\), without boundary zeros or poles. The Toolbox chain rule
transfers the measurement to the image contour:
\begin{equation}\label{eq:s3-argument}
 \softbox{C_{\Gamma,0}\!\left(\frac{F'}F\right)
 =C_{F\circ\Gamma,1}1(0)
 =\operatorname{Ind}(F\circ\Gamma,0)
 =N_U(F)-P_U(F).}
\end{equation}
Injectivity is unnecessary. Locally, \(F=(z-a)^m h\), \(h(a)\ne0\),
gives \(F'/F=m/(z-a)+h'/h\); poles give negative orders.
Excision and invariance assemble these local measurements into the
last equality.

The change of variable reveals what the argument principle measures:
the zero--pole count is the Cauchy convolution of the constant density
along the image contour, evaluated at the origin. Condition (H),
applied to the constant density on \(\mathbb C\), makes this measurement
depend only on the homotopy class of the image contour in
\(\mathbb C\setminus\{0\}\).
\quietemphasis{The same operation reads multiplicities in the source
and a homotopy-invariant measurement in the punctured image plane.}
\Needspace{9\baselineskip}
\medskip\noindent\textbf{Rouch\'e: admissible perturbation.}
For holomorphic \(f,g\) near \(\overline U\) with \(|g|<|f|\) on
\(\Gamma\), put \(F_s=f+sg\), \(0\leq s\leq1\). The source contour
stays fixed, but its image moves through
\(H(s,t)=F_s(\Gamma(t))\). The Toolbox estimate
\[
 |F_s|\geq|f|-|g|>0\quad\text{on }\Gamma
\]
makes \(H\) an admissible homotopy in \(\mathbb C\setminus\{0\}\).
Condition (H), applied to the constant density on \(\mathbb C\), with
measurement at zero and admissible contours in \(\mathbb C\setminus\{0\}\), preserves \(C_{F_s\circ\Gamma,1}1(0)\).
By \eqref{eq:s3-argument} this is precisely the zero count, so
\[
 \plainbox{N_U(f)=N_U(f+g).}
\]
\emph{Change of variable identifies the measurement; perturbation preserves
its admissible homotopy class.} This gives a unified reading of the
classical arguments, rather than a new proof of Rouch\'e; compare
\cite[\S10]{Romik2021}.

The count remains fixed, but the zeros can move. Weighted measurements
record that additional information. For a holomorphic,
boundary-nonvanishing family \(F_s\), depending \(C^1\) on a real parameter
on a common open neighborhood of \(\overline U\),
and a weight \(\psi\) holomorphic near \(\overline U\), put
\[
 M_\psi(s)=C_{\Gamma,0}\!\left(\psi\frac{F_s'}{F_s}\right)
 =\sum_j m_j(s)\psi(a_j(s)).
\]
Here \(a_j(s)\) are the distinct zeros inside \(U\) and \(m_j(s)\)
their multiplicities; smooth labelling of individual roots is unnecessary.
The weight \(1\) gives their count, while \(z^k\) gives their power sums.
Differentiation of the contour expression and integration of a total derivative give
\begin{equation}\label{eq:s3-moments}
 M_\psi'(s)=-C_{\Gamma,0}\!\left(\psi'\frac{\partial_sF_s}{F_s}\right).
\end{equation}
For \(F_s(z)=z^2-s\) on the unit disk, \(|s|<1\), the weights
\(1,z,z^2\) give \(2,0,2s\): smooth measurements even at the collision
of roots. For constant \(\psi\), the right side of
\eqref{eq:s3-moments} vanishes, recovering invariance of the count;
nonconstant weights measure its finer variation. Power-sum contour
measurements are also used to locate zeros numerically
\cite{DelvesLyness1967}.

\Needspace{9\baselineskip}
\medskip\noindent\textbf{Cauchy--Pompeiu: the obstruction to invariance.}
Our starting question remains informative when invariance fails.
Let \(f\in C^1\) and set
\(\bar\partial f=\partial f/\partial\bar z=\tfrac12(f_x+if_y)\).
For a smooth homotopy of closed contours \(H(s,t)\) avoiding a fixed
\(z\), differentiation and integration by parts yield
\begin{equation}\label{eq:s3-homotopy-derivative}
 \frac{d}{ds}C_{H(s,\cdot)}f(z)=\frac1{2\pi i}\int
 \frac{\bar\partial f(H)}{H-z}
 \big(\overline{H_s}H_t-\overline{H_t}H_s\big)\,dt.
\end{equation}
Here subscripts on \(H\) inside the integral denote partial derivatives.
Integration by parts in \(t\) cancels all terms containing
\(\partial f\), leaving precisely the \(\bar\partial f\) contribution.
For nested contours bounding a swept region \(S\), the integrated form is
\begin{equation}\label{eq:s3-deformation-defect}
 C_{\gamma_1}f(z)-C_{\gamma_0}f(z)
 =\frac1\pi\iint_S\frac{\bar\partial f(\zeta)}{\zeta-z}\,dA(\zeta),
 \qquad \partial S=\gamma_1-\gamma_0.
\end{equation}
Here \(dA(\zeta)=d\xi\,d\eta\), for \(\zeta=\xi+i\eta\), is ordinary
planar area measure. General deformations count swept area with oriented
multiplicity.
Shrinking the inner circle to \(z\) gives, for a bounded \(C^1\) domain
\(U\) and \(f\in C^1\) near \(\overline U\), with \(z\in U\) and positively oriented boundary \(\partial U\),
\begin{equation}\label{eq:s3-reconstruction-defect}
 \plainbox{f(z)=C_{\partial U}f(z)
 -\frac1\pi\iint_U\frac{\bar\partial f(\zeta)}{\zeta-z}\,dA(\zeta).}
\end{equation}
The boundary term is holomorphic in \(U\); the area term corrects its
failure to reproduce \(f\). Thus \emph{the same \(\bar\partial f\)
measures failure of deformation invariance and failure of analytic
reconstruction}. Formula~\eqref{eq:s3-reconstruction-defect}, also called
the Cauchy--Green formula, is standard; see
\cite[Ch.~IV, \S8, p.~127, (8.5)]{Gamelin2001} and
\cite[\S4.1]{LeblSCV}.
A single zero measurement does not suffice: for \(f(\zeta)=\bar\zeta\),
every circle centered at zero gives \(C_\gamma f(0)=0\).

\smallskip\noindent\emph{Why conjugation recurs in the examples.}
Many familiar smooth, nonholomorphic functions are built from real and
imaginary parts or squared moduli. These operations already contain
conjugation:
\[
 \operatorname{Re}h=\tfrac12(h+\bar h),\qquad
 \operatorname{Im}h=\tfrac1{2i}(h-\bar h),\qquad |h|^2=h\bar h.
\]
For a complex-valued real \(C^1\) function,
\(df=f_z\,dz+f_{\bar z}\,d\bar z\); the second term measures the
failure of its differential to be complex linear. This explains the
recurrence of conjugation without claiming that its mere appearance in
a formula forces nonholomorphy. For example, \(|z|^2\) is complex
differentiable at zero, but holomorphic on no neighborhood of zero.

\smallskip\noindent\emph{The same distinction in power-series language.}
Holomorphic functions have local power-series resolutions. For a real
analytic function the corresponding expansion, with \(w=z-a\), is
\[
 f(a+w)=\sum_{m,n\geq0}c_{mn}w^m\bar w^n,\qquad
 f\text{ holomorphic near }a\ \Longleftrightarrow\ c_{mn}=0\ (n>0).
\]
General \(C^\infty\) functions need not have convergent Taylor expansions,
but their \(\bar\partial\)-defect remains defined. Thus the homotopy
calculation does more than restate an integral formula: it identifies
the same obstruction in deformation, reproduction, and the familiar
contrast between smooth and holomorphic examples. This is precisely the
area correction in \eqref{eq:s3-reconstruction-defect}.

\subsubsection{Riesz projections: the principle in operator theory}
The correspondence with the scalar construction is direct: the scalar
integrand becomes an operator-valued resolvent; its excluded singular set
becomes the spectrum; winding selects a spectral part; and integration
returns the projection onto the corresponding invariant subspace.
For a diagonalizable matrix, this simply keeps the components in the
selected eigenspaces and annihilates the others.

Let \(T\) be bounded on a complex Banach space \(X\), with
\(\sigma(T)=K_1\mathbin{\dot\cup}K_2\) a union of disjoint compact sets.
The resolvent \(R(\zeta,T)=(\zeta I-T)^{-1}\) is holomorphic outside
\(\sigma(T)\). Choose a finite contour cycle \(\Gamma\) there, of index
\(1\) on \(K_1\) and \(0\) on \(K_2\). Then
\begin{equation}\label{eq:s3-riesz}
 \Pi=\frac1{2\pi i}\int_\Gamma R(\zeta,T)\,d\zeta,
 \qquad \Pi^2=\Pi,\qquad T\Pi=\Pi T.
\end{equation}
The spectrum consists of those \(\zeta\) for which \(\zeta I-T\)
is not invertible. It is the excluded set: ``not crossing the spectrum'' means
that the deformed contours stay in the resolvent set. Cauchy invariance
makes the measurement independent of admissible contour representatives;
more generally it depends on the homology class, encoded here by the
indices on the spectral parts. The spectrum need not consist of isolated
points or finitely many holes.

The additional algebraic input
\[
 R(\zeta,T)R(\eta,T)
 =\frac{R(\eta,T)-R(\zeta,T)}{\zeta-\eta}
\]
reduces the double contour integral for \(\Pi^2\) to \(\Pi\).
Complementary spectral selections give \(I-\Pi\) and the invariant
decomposition \(X=\operatorname{ran}\Pi\oplus\ker\Pi\).
These are bounded projections, generally not orthogonal.
\emph{Contour invariance provides consistent selection; the resolvent
identity turns it into a projection.} This is a direct structural
continuation of the elementary viewpoint in an operator-valued setting.
See \cite[\S\S10.23--10.27]{Rudin1991},
\cite[Ch.~VII, Props.~4.6, 4.11 and (6.9)]{Conway1990}, and
\cite[VII.3.17, VII.3.20--21]{DunfordSchwartz1958};
\cite[\S1.3]{Haase2018} gives a concise treatment.

\subsubsection{Planar homology and cohomology: naming the measurements}
For analytic integrands, homotopy invariance and additivity make periods
homomorphisms \(\pi_1(D)\to(\mathbb C,+)\); hence they factor through
\(H_1(D;\mathbb Z)\). Integration forgets the order of concatenation
and retains winding multiplicities. On
\(D=\mathbb C\setminus\{a_1,\ldots,a_m\}\), let \(\gamma_j\) be
small positive circles about the punctures. The dictionary is
\begin{equation}\label{eq:s3-duality}
 \plainbox{\begin{gathered}
 \omega_j=\frac1{2\pi i}\frac{dz}{z-a_j},\qquad
 \int_{\gamma_k}\omega_j=\delta_{jk},\\
 [\gamma]=\sum_j\operatorname{Ind}(\gamma,a_j)[\gamma_j]
 \quad\text{in }H_1(D;\mathbb Z).
 \end{gathered}}
\end{equation}
The punctured plane retracts onto a bouquet of \(m\) circles:
\([\gamma_j]\) are homology generators and indices are their coordinates.
The classes \([\omega_j]\) form the dual basis of complex de Rham
cohomology. Indeed, subtracting
\(\sum_j(\int_{\gamma_j}\alpha)\omega_j\) from a smooth closed form
\(\alpha\) kills all periods, so path integration supplies a potential.
For analytic \(f\), the resulting coordinates are its residues:
\begin{equation}\label{eq:s3-residues}
 C_{\gamma,0}f=\sum_j\operatorname{Ind}(\gamma,a_j)\lambda_{-1}^{(a_j)}(f),
 \qquad
 \left[\frac{f\,dz}{2\pi i}\right]
 =\sum_j\lambda_{-1}^{(a_j)}(f)[\omega_j].
\end{equation}
Thus the loops, indices, kernels, and measurements already in use become,
respectively, cycles, coordinates, dual classes, and the integration
pairing. This is the precise conceptual preparation furnished by the
starting principle. See \cite[Chs.~3, 5--6]{Fulton1995} and
\cite[Ch.~16]{StewartTall2018}.

\subsubsection{A language already at the entrance to advanced theory}
Bell begins his second-course development with the Cauchy transform,
then introduces Hardy space, the Szeg\H{o} projection, and the
Kerzman--Stein formula \cite[Chs.~3--4, pp.~9--16]{Bell2016}.
The elementary correspondence is again concrete: boundary values are the
input, a holomorphic function inside the domain is the output, and the
Cauchy kernel performs reconstruction. For a smooth bounded planar domain,
the interior transform is holomorphic even for nonholomorphic boundary
data; reproduction holds for holomorphic traces. Such arbitrary boundary
data do not themselves carry a homotopy-invariance assertion. Where a
holomorphic extension is available, the contour invariance established
here explains the freedom to choose its admissible representatives. The Szeg\H{o} operator \(S\) is the orthogonal
projection in \(L^2(\partial U,ds)\) onto the Hardy space: the boundary functions obtainable as \(L^2\)
limits of smooth holomorphic traces. The Cauchy boundary projection need not
be orthogonal; their comparison is part of the advanced theory.

On the unit circle the connection with our component notation is literal:
\[
 Sh=\sum_{n\geq0}\widehat h(n)e^{int}\quad\text{in }L^2(\mathbb T),
 \qquad h_+(z)=\sum_{n\geq0}\widehat h(n)z^n\quad(|z|<1).
\]
Here \(h_+\) is the holomorphic extension of the projected boundary data.
Bell's Bergman space uses area measure instead: the orthogonal
projection onto \(A^2(U)=\mathcal O(U)\cap L^2(U,dA)\) is represented
by its reproducing kernel \cite[Ch.~15, pp.~69--74]{Bell2016}.
These Hilbert-space constructions require additional analytic structure.
Our point is the preparation already supplied by the elementary
language: \emph{integral operators, reproduction, and component
projections} are working objects from the outset, making this advanced
starting point immediately recognizable.

\subsection{Parallel Picture in Other Theories: Identity Resolution via One \mbox{Fundamental Action}}

We now exhibit a common pattern across several theories: start with a
``basic operation,'' identify the ``fundamental objects'' it characterizes,
and use the same operation to extract the components that assemble
into a resolution of the identity. The comparison concerns this double
role of the operation. The references to the author's companion manuscripts
provide context for the origin and further development of the ideas;
the arguments needed for the two convolution models are given in full below.

\medskip\noindent\textbf{The elementary model: a basis and an inner product.}
The first example illustrates the idea in its simplest setting: a
finite-dimensional inner-product space. The basis concept already contains reconstruction. An inner product adds
a direct means of recognizing suitable bases and calculating components.
In finite dimension a maximal orthonormal set \(\{v_j\}\) is a basis:
otherwise a nonzero vector in the orthogonal complement could be
normalized and adjoined. The same operation then gives
\[
 P_jv=\langle v,v_j\rangle v_j,\qquad I=\sum_jP_j.
\]
Maximal orthonormal sets are thus the fundamental families, characterized
and used through the inner product; no canonical choice is asserted.

\medskip\noindent\textbf{Regular Sturm--Liouville theory.}
The same picture persists with a differential operator selecting the
family. An integrating factor brings the equation into formally symmetric
weighted form; appropriate self-adjoint boundary conditions specify its
realization in \(L^2_w\). For a regular problem the normalized eigenfunctions
\(\phi_j\) form an orthonormal basis, and
\[
 P_jf=\langle f,\phi_j\rangle_w\phi_j,\qquad
 f=\sum_jP_jf,\qquad Lf=\sum_j\mu_jP_jf\quad(f\in\mathcal D(L)).
\]
The last two sums converge in \(L^2_w\). The weight, domain, and boundary
conditions belong to the selecting data; formal symmetry alone is
insufficient.

\medskip\noindent\textbf{Selection by the operation.}
For the two convolution models below, the selecting condition can be
stated without the general framework of \cite{BenSimonInPreparation}.
In a commutative semigroup \((S,\star)\), write
\(B\star B=\{b\star c:b,c\in B\}\). We select nonempty sets \(B\)
that satisfy \(B\star B=B\) and are minimal under inclusion among
nonempty sets satisfying this equality. This is an additional algebraic
condition, not a consequence of homotopy invariance. The following
argument makes the selection explicit.

\begin{hublemma}\label{lem:s3-minimal-selection}
In a commutative semigroup, the nonempty inclusion-minimal sets
satisfying \(B\star B=B\) are exactly the singletons \(\{p\}\)
with \(p\star p=p\).
\end{hublemma}
\begin{proof}
Write products multiplicatively. For each positive integer \(m\), the
set \(B^{[m]}=\{b^m:b\in B\}\) is a nonempty subset of \(B\), and
commutativity gives
\((B^{[m]})^2=(B^2)^{[m]}=B^{[m]}\).
Minimality therefore gives \(B^{[m]}=B\): every element has an
\(m\)-th root in \(B\).
Choose \(b_0\in B\) and successive square roots
\(b_{j+1}^2=b_j\). The set
\[
 D=\{b_j^q:j\geq0,\ q\geq1\}
\]
satisfies \(D^2=D\): two elements are powers of a common later root,
and \(b_j^q=(b_{j+1}^q)^2\). Hence \(D=B\).
A cube root of \(b_0\) thus has the form \(b_j^q\), so
\(b_j^{3q}=b_0=b_j^{2^j}\), with \(3q\ne2^j\).
If \(u^p=u^r\), \(p>r\), then
\(u^{n+p-r}=u^n\) for all \(n\geq r\).
Choose \(k\geq r\) divisible by \(p-r\); repeated use of this
identity gives \(u^{2k}=u^k\). Applying it to \(u=b_j\) produces
an idempotent in \(B\). Its singleton satisfies the same equality,
so minimality makes it all of \(B\). The converse is immediate.
\end{proof}

In an algebra, call a nonzero idempotent \emph{primitive} if it cannot
be written as a sum of two nonzero idempotents whose product is zero.
The lemma selects idempotent blocks; this further algebraic distinction
will identify the individual modes. Neither step presupposes a choice
of basis.
\quietemphasis{In each convolution model, the operation and the stated
selection conditions determine exactly the fundamental sets and their
primitive building blocks.}

\medskip\noindent\textbf{Periodic convolution: an intrinsic selection.}
Use normalized measure \(dm=dt/(2\pi)\) on \(\mathbb T\), and put
\[
 (f*g)(\theta)=\int_{\mathbb T}f(\theta-t)g(t)\,dm(t),
 \qquad f,g\in L^2(\mathbb T).
\]
Here functions are identified almost everywhere.
Cauchy--Schwarz bounds \(|f*g|\) by \(\|f\|_2\|g\|_2\), so the
product lies in \(L^2\). Since \(L^2\subset L^1\) on this probability
space, substitutions and Fubini give associativity, commutativity, and
\[
 \widehat{f*g}(n)=\widehat f(n)\widehat g(n),\qquad
 \widehat f(n)=\int_{\mathbb T}f(t)e^{-int}\,dm(t).
\]

\smallskip\noindent\emph{Uniqueness and convergence.}
For completeness, the needed Fourier uniqueness follows from the
nonnegative kernels
\[
 K_N(t)=\frac1{N+1}\left|\sum_{j=0}^N e^{ijt}\right|^2
 =\sum_{|n|\leq N}\left(1-\frac{|n|}{N+1}\right)e^{int}.
\]
They have integral one, and their mass outside any neighborhood of zero
tends to zero by the finite geometric-sum formula. Translation is
continuous in \(L^2\) (first for continuous functions, then by density),
so
\[
 \|f*K_N-f\|_2\leq
 \int_{\mathbb T}K_N(t)\|f(\,\cdot-t)-f\|_2\,dm(t)
 \longrightarrow0.
\]
Thus all Fourier coefficients zero implies \(f=0\); moreover,
trigonometric polynomials are dense in \(L^2\).

\smallskip\noindent\emph{Selected blocks and individual modes.}
If \(p*p=p\), its coefficients are zero or one. Bessel's inequality
makes their support \(E\subset\mathbb Z\) finite, and uniqueness gives
\(p=p_E:=\sum_{n\in E}e^{int}\). Conversely every such polynomial
is idempotent. Lemma~\ref{lem:s3-minimal-selection} therefore classifies
the selected sets as precisely \(\{p_E\}\), including
\(E=\varnothing\), which gives \(\{0\}\).
Since \(p_E*p_F=p_{E\cap F}\), a block with at least two frequencies
splits into two nonzero orthogonal idempotents. A singleton block cannot
so split. The primitive nonzero idempotents are consequently exactly
\(\varepsilon_n(t)=e^{int}\).

\smallskip\noindent\emph{Reconstruction by the same action.}
A substitution gives \(f*\varepsilon_n=\widehat f(n)\varepsilon_n\), hence
\begin{equation}\label{eq:s3-fourier}
 f*p_E=\sum_{n\in E}\widehat f(n)\varepsilon_n,\qquad
 f*p_{\{-N,\ldots,N\}}\longrightarrow f\quad\text{in }L^2.
\end{equation}
Indeed, the latter sums are orthogonal projections onto the increasing
finite spans of the modes, whose density was just established.
The classification selects the kernels; the convergence argument
justifies reconstruction from their labelled mixed products with \(f\).

\medskip\noindent\textbf{The complex counterpart: selection within a contour algebra.}
This model starts independently with continuous functions on
\(\Omega_b=\{e^{-b}<|z|<e^b\}\), \(b>0\), and
\begin{equation}\label{eq:s3-diamond}
 (f\diamond_\gamma g)(z)=\frac1{2\pi i}\int_\gamma
 f(z/w)g(w)\,\frac{dw}{w},\qquad
 w,z/w\in\Omega_b,\quad\operatorname{Ind}(\gamma,0)=1.
\end{equation}
Let \(\mathcal H_b\subset C(\Omega_b)\) consist of the functions for which
\((1\diamond_\gamma f)(1)\) is invariant under homotopies of winding-one
loops in \(\Omega_b\). At \(z=1\), the condition \(1/w\in\Omega_b\)
is automatic. The structural conclusions needed here follow from the
results already proved in this paper.

\smallskip\noindent\emph{From the single probe to Laurent representation.}
Attach a small triangular loop, using a path traversed in both directions,
to a winding-one loop. Contracting the triangle is an admissible homotopy;
the connecting integrals cancel. Probe invariance therefore implies
\(\int_{\partial\Delta}f(w)\,dw/w=0\) for every triangle contained in a
small disk in \(\Omega_b\). Morera's theorem proved above makes
\(f(w)/w\), and hence \(f\), holomorphic. Conversely, for holomorphic
\(f\), local primitives of \(f(w)/w\) give the required invariance.
Thus \(\mathcal H_b\) is exactly the holomorphic class on \(\Omega_b\),
now obtained from continuity and the probe condition.
Annular
reconstruction supplies
\[
 f(z)=\sum_{n\in\mathbb Z}c_n(f)z^n
 \quad\text{normally on compact subannuli}.
\]
For \(0<c<b\), coefficient estimates on \(|z|=e^{\pm c}\) give
\( |c_n(f)|\leq M_c(f)e^{-c|n|}\), where \(M_c(f)\) is the maximum
of \(|f|\) on those two circles.

\smallskip\noindent\emph{The product and its selected objects.}
For fixed \(z\in\Omega_b\), the admissible \(w\)-region in
\eqref{eq:s3-diamond} is a nonempty annulus: the circle
\(|w|=\sqrt{|z|}\) lies in it. The integrand is holomorphic there,
so the value is independent of the winding-one representative.
On that circle both Laurent series converge absolutely and uniformly;
termwise integration yields
\[
 (f\diamond g)(z)=\sum_{n\in\mathbb Z}c_n(f)c_n(g)z^n.
\]
This is the coefficientwise Hadamard product of the Laurent series;
for the classical power-series product and its multiplicative contour
representation, see~\cite[\S1]{DubussySchneiders2022}.
The coefficient estimates above make this series normally convergent
on \(\Omega_b\) (in fact on \(\Omega_{2b}\)). Thus \(\mathcal H_b\)
is closed under the product, and uniqueness of Laurent coefficients
makes it a commutative associative algebra.

An idempotent has coefficients in \(\{0,1\}\); their exponential
decay forces finite support. The idempotents are therefore exactly
\(p_E(z)=\sum_{n\in E}z^n\), with \(E\subset\mathbb Z\) finite.
Lemma~\ref{lem:s3-minimal-selection} identifies the selected sets as
\(\{p_E\}\). As in the preceding model,
\(p_E\diamond p_F=p_{E\cap F}\), so the primitive nonzero
idempotents are precisely the individual powers \(z^n\).

\smallskip\noindent\emph{Reconstruction by the same action.}
The product formula now gives
\[
 f\diamond z^n=c_n(f)z^n,\qquad
 f\diamond p_{\{-N,\ldots,N\}}\longrightarrow f
 \quad\text{locally uniformly on }\Omega_b.
\]
Explicitly, for \(0<d<c<b\) and \(e^{-d}\leq|z|\leq e^d\), the
truncation error is at most
\[
 \frac{2M_c(f)e^{-(c-d)(N+1)}}{1-e^{-(c-d)}}.
\]
Continuity and probe invariance prepare the analytic class; minimal
self-reproduction selects finite blocks; primitivity selects modes;
and normal convergence reconstructs the target. These are distinct steps,
all justified here. The companion investigation
\cite{BenSimonLaurent}, prompted by this example, develops the model
further. Table~\ref{tab:s3-comparison} collects the structural parallel;
the two convolution models retain their distinct analytic settings.

\begin{table}[H]
\centering\small
\arrayrulecolor{gray!55}
\renewcommand{\arraystretch}{1.25}
\setlength{\tabcolsep}{5pt}
\begin{tabularx}{\textwidth}{|>{\raggedright\arraybackslash}p{2.6cm}|*{3}{>{\raggedright\arraybackslash}X|}}
\hline
\textbf{Setting} & \textbf{Basic action} & \textbf{Fundamental objects} & \textbf{Components and resolution of the identity}\\\hline
\textbf{Finite-dimensional linear algebra}
& \(\langle v,w\rangle\); maximal orthonormal family
& \(v_j\)
& \(\langle v,v_j\rangle v_j\); finite sum\\\hline
\textbf{Regular Sturm--Liouville theory}
& \(\langle f,g\rangle_w\); self-adjoint operator and boundary data
& Eigenfunctions \(\phi_j\)
& \(\langle f,\phi_j\rangle_w\phi_j\); convergence in \(L^2_w\)\\\hline
\textbf{Periodic Fourier analysis}
& \(f*g\); minimal self-reproducing sets
& Primitive idempotents \(e^{int}\)
& \(f*e^{int}\); reconstruction in \(L^2\)\\\hline
\textbf{Complex analysis}
& \(K*_{[\gamma]}f\); also \(\diamond\) with the stated selection
& \((z-a)^n\); \(z^n\) selected by \(\diamond\)
& \(P_n^{(a)}f\); locally uniform reconstruction\\\hline
\end{tabularx}
\caption{One action, fundamental objects, and a resolution of the identity.}
\label{tab:s3-comparison}
\end{table}

Table~\ref{tab:s3-comparison} collects the comparison. A more general
formulation of this pattern is in preparation.
D'Angelo \cite{DAngelo2017} also brings linear, complex, Fourier, and
spectral analysis into a unified presentation, with substantial emphasis
on techniques and applications. The organizing comparison here is more
specific: a single action identifies fundamental objects and then uses
them to resolve the identity. The distinction lies in this organizing mechanism and its context.
The comparison is not simply the existence of bases in linear spaces:
convolution is not an inner product, nor is it supplied by linearity alone.

\subsection{Meeting points with other presentations}

Three overlapping classical approaches clarify the meeting points
with our proposed starting principle.

\medskip\noindent\textbf{Power series first.}
Power series provide the starting point in classical texts such as those
of Cartan and Lang \cite[Preface and Ch.~I]{Cartan1963},
\cite[Chs.~II--III]{Lang1999}. Romik's modern lecture notes likewise
place power series before contour integration \cite[\S\S4--5]{Romik2021}.
Our construction explains the foundational strength of this choice:
invariance yields almost immediately the resolution of the identity,
making coefficient calculus available for both computation and theory.
There is an essential difference in the entry point: we assume only
continuity of the density and condition (H), and derive analyticity.
The power-series route starts with analytic representation itself.
Appendix~\ref{app:foundational-map} locates this alternative entry at the
resolution of the identity and explains how it recovers the foundational
results. On an annulus, the three contours in
Figure~\ref{fig:s3-annulus} and the same geometric expansion supply the
analytic and singular components; that extension is already built into
the invariant operation.

\medskip\noindent\textbf{Geometry as explanation.}
As explained at the foundational junction in
Section~\ref{sec:cauchy-convolutional-hub} and displayed in
Appendix~\ref{app:foundational-map}, Cauchy's Green-theorem argument and
Goursat's subdivision proof meet condition (H) at the beginning of the
theory. Geometry establishes invariance from differentiability; our
starting point develops the structure forced by invariance. This early
meeting explains the sustained dialogue between the two viewpoints.
Building on a long geometric tradition, Needham's
\emph{Visual Complex Analysis} \cite{Needham2023} offers a particularly rich and
systematic realization of this viewpoint.

To glimpse one central idea: consider the interpretation of a
complex integral through circulation and flux; compare Needham's
presentation \cite[\S\S11.2.1, 11.2.4, pp.~547--553]{Needham2023}.
For \(h=u+iv\), the P\'olya field \(\mathbf P_h=(u,-v)\) gives
\[
 \int_\gamma h\,dz=\int_\gamma\mathbf P_h\cdot\mathbf T\,ds
 +i\int_\gamma\mathbf P_h\cdot\mathbf n_R\,ds,
\]
with right normal \(\mathbf n_R\), outward on a positive Jordan boundary.
For a holomorphic density, the Cauchy--Riemann equations give
\[
 \operatorname{div}\mathbf P_h=u_x-v_y=0,\qquad
 \operatorname{curl}\mathbf P_h=-v_x-u_y=0.
\]
Thus both flux and circulation are preserved under admissible deformation.
For \(h=1/z\), the field is radial: \((x,y)/(x^2+y^2)\).
A positive circle has circulation zero and flux \(2\pi\).
For \(h=i/z\), the field is clockwise tangential:
\((y,-x)/(x^2+y^2)\), with circulation \(-2\pi\) and flux zero.
The complex measurement records both in one operation.

\begin{figure}[H]
\centering
\begin{tikzpicture}[scale=0.83,>=stealth]
\foreach \shift in {0,7.1} {
 \begin{scope}[xshift=\shift cm]
 \draw[thick] (0,0) circle (1.6);
 \draw[->,thick] (40:1.6) arc (40:65:1.6);
 \fill (0,0) circle (1.5pt) node[below right] {$0$};
 \end{scope}}
\foreach \ang in {0,30,...,330} {
 \foreach \rr in {0.65,1.2,1.9} {
  \pgfmathsetmacro{\endrr}{\rr+0.22/\rr}
  \draw[->,gray!75] (\ang:\rr)--(\ang:\endrr);
  \begin{scope}[xshift=7.1cm]
   \draw[->,gray!75] (\ang:\rr)--++({0.22*sin(\ang)/\rr},{-0.22*cos(\ang)/\rr});
  \end{scope}
 }}
\node at (0,2.55) {$h(z)=1/z$};
\node at (7.1,2.55) {$h(z)=i/z$};
\node[align=center] at (0,-2.6) {circulation $=0$\\flux $=2\pi$};
\node[align=center] at (7.1,-2.6) {circulation $=-2\pi$\\flux $=0$};
\end{tikzpicture}
\caption{Two P\'olya fields illustrating the real and imaginary measurements
of a complex integral. Deformations avoiding zero preserve both measurements.}
\label{fig:s3-polya}
\end{figure}
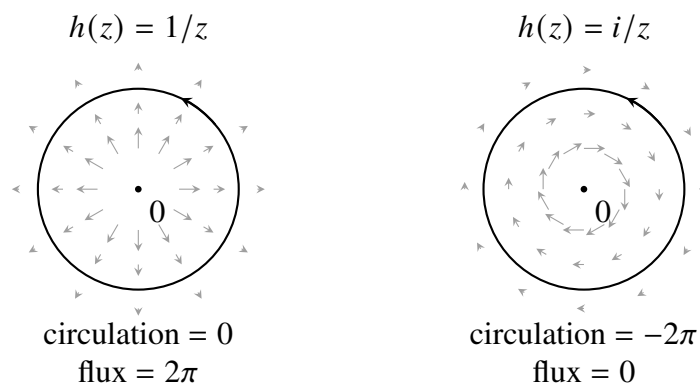
This is one illustration of a much broader geometric interpretation,
which includes winding number and many further results.

\medskip\noindent\textbf{Complex differentiability first.}
The systematic development from complex numbers through differentiability
to the integral theorems is represented, for example, by
\cite{Ahlfors1979,BakNewman2010,SaffSnider2003,Isaev2017},
Stewart--Tall and Remmert~\cite{StewartTall2018,Remmert1991}, and
Borcherds's lectures~\cite{BorcherdsLectures}. It presents complex analysis
as a self-contained theory developed from the beginning. Its strength lies in this
completeness and orderly progression; the central organizing
structure may, however, become visible only later in the development.
The meeting point is where complex differentiability yields the integral
theorems and homotopy invariance. So
here the invariant operation is obtained from differentiability;
ours investigates what it forces on continuous functions.

\section{Concluding perspective and possible teaching routes}
\label{sec:conclusion}

Sections~2 and~3 develop a single structural assumption---homotopy invariance of Cauchy convolution integrals for continuous densities---into an organizing principle for elementary complex analysis. It yields Cauchy reproduction, analyticity, the derivative formulas, and a resolution of the identity; the same reconstruction mechanism proves Morera's theorem and gives the Cauchy--Goursat vanishing conclusion. Selected examples then explore the reach of this viewpoint across classical applications, connections to advanced theory, and structural parallels with linear algebra, Fourier analysis, and Sturm--Liouville theory. This development also reveals foundational meeting points with the geometric approach exemplified by Needham and with power-series-first presentations, as summarized in Appendix~\ref{app:foundational-map}. Throughout, the study of Cauchy convolution on homotopy classes and the Toolbox work in symbiosis: the Toolbox makes the construction operational, while the organizing principle gives its familiar operations a coherent analytic purpose. The symbolic forms and operational rules inherited from real-variable calculus are made precise in Appendix~\ref{app:toolbox}.

\medskip\noindent\textbf{Can this viewpoint guide teaching?}
A cautious affirmative answer suggests two possible teaching directions.
We sketch guiding principles from which course outlines might be developed,
without entering into syllabus details.

\Needspace{12\baselineskip}
\medskip\noindent\textbf{A course centered on calculation.}
\begin{enumerate}[leftmargin=2em,itemsep=7pt,topsep=6pt]
\item \textbf{Operational fluency through the Toolbox.}
The first stage gives deliberate priority to symbolic fluency. With
tutorial support, students learn to operate the Toolbox by following the
real-variable parallels in Appendix~\ref{app:toolbox}: compute limits,
differentiate elementary expressions, apply the chain rule, choose useful
parametrizations, and calculate integrals. At this stage the instruction
is direct: identify the operation, consult its parallel form, and carry
it out under the stated hypotheses. Structural depth is deferred rather
than required before calculation can begin.

\item \textbf{Invariance and reconstruction.}
The second stage states invariance of the Cauchy convolution under
admissible contour deformation; the word ``homotopy'' need not be
used. The resolution of the identity receives one of the course's few
complete proofs, first on a disk and then on an annulus using the
three-contour picture.

\item \textbf{Calculation through components.}
Thereafter, local holomorphic functions are handled
through these Taylor and Laurent resolutions. Orders of zeros and poles,
residues, and isolated-singularity types become component data. In the
residue theorem, invariance assembles the local data with the appropriate
winding numbers; integral evaluations add contour choices, estimates,
and limit operations from the Toolbox.

\end{enumerate}

\medskip\noindent\textbf{A course centered on structural understanding.}
\begin{enumerate}[leftmargin=2em,itemsep=7pt,topsep=6pt]
\item \textbf{Examining the Toolbox.}
For a mathematically strong class, begin with the same correspondences
but demand more of them. Students can examine their hypotheses, compare
geometric representations, and investigate how the nonuniqueness of
polar coordinates leads to branch questions.

\item \textbf{Building the foundations from invariance.}
Section~2 then provides the
full construction from invariant convolution. The resolution of the
identity remains a working tool, while the operation on homotopy classes
stays explicit throughout subsequent arguments.

\item \textbf{Developing conceptual understanding through the organizing principle.}
Students use the organizing principle to understand how the concepts and
arguments of complex analysis fit together. For example, the argument principle
becomes change of variable; Rouch\'e becomes admissible perturbation;
periods reveal the passage to homology; and Cauchy--Pompeiu measures the
obstruction to invariance for a nonholomorphic density. Functions
involving conjugation, modulus, and real or imaginary parts provide
contrasting test objects, with the appropriate regularity checked before
applying the defect formula. The class thus learns both how to calculate
and why the same operation continues to organize the theory.
\end{enumerate}

In the technical outline, the structure guides the selection and order
of the tools. In the advanced outline, that structure also remains an
explicit object of investigation throughout the course.

\Needspace{9\baselineskip}
\section*{Declaration of generative AI use}
After the substantive content was fully authored, OpenAI ChatGPT tools were utilized to enhance readability, suggest structural flow improvements, improve language, proofread the manuscript, assist with LaTeX and code preparation, and conduct literature and source searches concerning prior work and originality. All AI-assisted material was independently reviewed and, where necessary, revised by the author.

\appendix
\section{The Real Conceptual Toolbox of Complex Analysis}
\label{app:toolbox}
\par\addvspace{8pt}
{\large\itshape\noindent
The reader will recognize nearly every entry; that is precisely the point.
The tables distinguish the familiar tools from the structural principle
that gives their use its complex-analytic force. The Toolbox is auxiliary,
but indispensable.\par}
\addvspace{10pt}

The tables display the inherited symbolic calculus used throughout the
paper. The first two compare scalar functions; the third compares planar
vector fields with complex functions. Here \(I\subset\mathbb R\) is an
open interval, \(D\subset\mathbb C\) a domain, and integrands are
continuous on their integration paths. Paths are piecewise \(C^1\).
Increments in a tagged sum are differences of consecutive endpoints;
\(|P|\to0\) means that the largest subinterval or subarc length tends
to zero. Equivalently, refine a fixed piecewise \(C^1\) parametrization.

The comparison concerns the form of definitions and operations under
their stated hypotheses. It does not identify the implications of those
definitions: the rigidity forced by complex differentiability is precisely
what the main text develops. Repeated integral rules are intentional,
since they are visible from both the scalar and the line-integral viewpoint.

\begingroup
\fontsize{10.5}{12.4}\selectfont
\setlength{\tabcolsep}{6pt}
\renewcommand{\arraystretch}{1.1}
\arrayrulecolor{gray!55}
\newlength{\toolcol}
\setlength{\toolcol}{\dimexpr(\textwidth-6\tabcolsep-4\arrayrulewidth-2.6cm)/2\relax}
\setlength{\LTpre}{8pt}
\setlength{\LTpost}{10pt}
\setlength{\abovedisplayskip}{5pt}
\setlength{\belowdisplayskip}{5pt}
\begin{longtable}{|>{\raggedright\arraybackslash}p{2.6cm}|*{2}{>{\raggedright\arraybackslash}p{\toolcol}|}}
\caption{Scalar calculus: definitions and operations.}\label{tab:s4-scalar}\\
\hline
\rowcolor{gray!7}\textbf{Property} & \textbf{Complex: \(z_0,f(z)\)} & \textbf{Real: \(x_0,f(x)\)}\\\hline
\endfirsthead
\multicolumn{3}{l}{\textbf{Table \thetable\ (continued). Scalar calculus}}\\[3pt]
\hline\rowcolor{gray!7}\textbf{Property} & \textbf{Complex} & \textbf{Real}\\\hline
\endhead
\endfoot
\endlastfoot
\textbf{1. Limit}
& \(\displaystyle\lim_{z\to z_0}f(z)=L\).
& \(\displaystyle\lim_{x\to x_0}f(x)=L\).\\\hline
\textbf{Limit arithmetic}
& If \(f(z)\to A\), \(g(z)\to B\), then\par
\(\alpha f+\beta g\to\alpha A+\beta B\),\par
\(fg\to AB\), \(f/g\to A/B\) if \(B\ne0\).
& If \(f(x)\to A\), \(g(x)\to B\), then\par
\(\alpha f+\beta g\to\alpha A+\beta B\),\par
\(fg\to AB\), \(f/g\to A/B\) if \(B\ne0\).\\\hline
\textbf{Composition of limits}
& If \(f(z)\to b\) and \(g\) is continuous at \(b\), then \(g(f(z))\to g(b)\).
& If \(f(x)\to b\) and \(g\) is continuous at \(b\), then \(g(f(x))\to g(b)\).\\\hline
\textbf{2. Continuity at a point}
& \(\displaystyle\lim_{z\to z_0}f(z)=f(z_0)\).
& \(\displaystyle\lim_{x\to x_0}f(x)=f(x_0)\).\\\hline
\textbf{3. Continuity on an open set}
& \(f\in C(D)\).
& \(f\in C(I)\).\\\hline
\textbf{Continuity calculus}
& For continuous \(f,g\): \(\alpha f+\beta g\), \(fg\), and \(f/g\) where \(g\ne0\) are continuous. So is \(g\circ f\) where defined, with \(g\) continuous on the range of \(f\).
& For continuous \(f,g\): \(\alpha f+\beta g\), \(fg\), and \(f/g\) where \(g\ne0\) are continuous. So is \(g\circ f\) where defined, with \(g\) continuous on the range of \(f\).\\\hline
\textbf{4. Derivative at a point}
& \(\displaystyle f'(z_0)=\lim_{z\to z_0}\frac{f(z)-f(z_0)}{z-z_0}\).
& \(\displaystyle f'(x_0)=\lim_{x\to x_0}\frac{f(x)-f(x_0)}{x-x_0}\).\\\hline
\textbf{5. Differentiability on an open set}
& \(f:D\to\mathbb C\), differentiable at every \(z\in D\).
& \(f:I\to\mathbb R\), differentiable at every \(x\in I\).\\\hline
\textbf{6. Arithmetic and chain rule}
& \(\begin{aligned}(f\pm g)'&=f'\pm g',\\(fg)'&=f'g+fg',\\(f/g)'&=\frac{f'g-fg'}{g^2},\\(g\circ f)'&=(g'\circ f)f'.\end{aligned}\)\par
At \(z\); derivatives exist, \(g\ne0\), composition defined.
& \(\begin{aligned}(f\pm g)'&=f'\pm g',\\(fg)'&=f'g+fg',\\(f/g)'&=\frac{f'g-fg'}{g^2},\\(g\circ f)'&=(g'\circ f)f'.\end{aligned}\)\par
At \(x\); derivatives exist, \(g\ne0\), composition defined.\\\hline
\textbf{7. Linear approximation}
& \(f'(z_0)\) exists iff, for some \(m\in\mathbb C\),\par
\(\begin{aligned} f(z)-f(z_0)&=m(z-z_0)\\&\quad+o(|z-z_0|).\end{aligned}\)\par
Necessarily \(m=f'(z_0)\).
& \(f'(x_0)\) exists iff, for some \(m\in\mathbb R\),\par
\(\begin{aligned} f(x)-f(x_0)&=m(x-x_0)\\&\quad+o(|x-x_0|).\end{aligned}\)\par
Necessarily \(m=f'(x_0)\).\\\hline
\textbf{8. Integral as a limit of sums}
& \(\displaystyle\int_\gamma f(z)\,dz\)\par
\(\displaystyle\quad=\lim_{|P|\to0}\sum_j f(\zeta_j)\Delta z_j\).\par
\(\Delta z_j=z_j-z_{j-1}\); the tag \(\zeta_j\) lies on the subarc.
& \(\displaystyle\int_a^b f(x)\,dx\)\par
\(\displaystyle\quad=\lim_{|P|\to0}\sum_j f(\xi_j)\Delta x_j\).\par
\(\Delta x_j=x_j-x_{j-1}\); the tag \(\xi_j\) lies in the subinterval.\\\hline
\textbf{Linearity}
& \(\displaystyle\int_\gamma(\alpha f+\beta g)\,dz\)\par
\(\displaystyle=\alpha\int_\gamma f\,dz+\beta\int_\gamma g\,dz\).
& \(\displaystyle\int_a^b(\alpha f+\beta g)\,dx\)\par
\(\displaystyle=\alpha\int_a^b f\,dx+\beta\int_a^b g\,dx\).\\\hline
\textbf{Reversal}
& \(\displaystyle\int_{\gamma^{-1}}f\,dz=-\int_\gamma f\,dz\).
& \(\displaystyle\int_b^a f\,dx=-\int_a^b f\,dx\).\\\hline
\textbf{Substitution via the chain rule}
& \(\displaystyle\int_{\phi\circ\gamma} f(w)\,dw\)\par
\(\displaystyle=\int_\gamma f(\phi(z))\phi'(z)\,dz\),\par
for \(\phi\) with continuous complex derivative near \(\gamma\).
& \(\displaystyle\int_{\phi(a)}^{\phi(b)}f(u)\,du\)\par
\(\displaystyle=\int_a^b f(\phi(x))\phi'(x)\,dx\),\par
for \(\phi\in C^1\); \(f\) continuous on an interval containing \(\phi([a,b])\).\\\hline
\textbf{Integration by parts}
& \(\displaystyle\int_\gamma fg'\,dz\)\par
\(\displaystyle=[fg]_\gamma-\int_\gamma f'g\,dz\),\par
where \([fg]_\gamma=fg(\gamma(b))-fg(\gamma(a))\); derivatives are continuous near the path.
& \(\displaystyle\int_a^b fg'\,dx\)\par
\(\displaystyle=[fg]_a^b-\int_a^b f'g\,dx\),\par
for \(f,g\in C^1([a,b])\).\\\hline
\textbf{Newton--Leibniz, given a primitive}
& If \(F'=f\) near \(\gamma\),\par
\(\displaystyle\int_\gamma f\,dz=F(\gamma(b))-F(\gamma(a))\).
& If \(F'=f\) on \([a,b]\),\par
\(\displaystyle\int_a^b f\,dx=F(b)-F(a)\).\\\hline
\end{longtable}

\Needspace{25\baselineskip}
\begin{longtable}{|>{\raggedright\arraybackslash}p{2.6cm}|*{2}{>{\raggedright\arraybackslash}p{\toolcol}|}}
\caption{Scalar power series and termwise calculus.}\label{tab:s4-series}\\
\hline\rowcolor{gray!7}\textbf{Property} & \textbf{Complex: \(z_0,S(z)\)} & \textbf{Real: \(x_0,S(x)\)}\\\hline
\endfirsthead
\hline\rowcolor{gray!7}\textbf{Property} & \textbf{Complex} & \textbf{Real}\\\hline
\endhead
\endfoot\endlastfoot
\textbf{Formal definition}
& \(\displaystyle\sum_{n\geq0}a_n(z-z_0)^n\), \(a_n\in\mathbb C\).
& \(\displaystyle\sum_{n\geq0}a_n(x-x_0)^n\), \(a_n\in\mathbb R\).\\\hline
\textbf{Radius of convergence}
& There is \(R\in[0,\infty]\): absolute convergence for \(|z-z_0|<R\), divergence for \(|z-z_0|>R\).
& There is \(R\in[0,\infty]\): absolute convergence for \(|x-x_0|<R\), divergence for \(|x-x_0|>R\).\\\hline
\textbf{Continuity of the sum}
& For \(0<r<R\), uniform convergence on \(|z-z_0|\leq r\); the sum \(S\) is continuous inside the convergence disk.
& For \(0<r<R\), uniform convergence on \(|x-x_0|\leq r\); the sum \(S\) is continuous inside the convergence interval.\\\hline
\textbf{Termwise integration}
& \(\displaystyle\int_{[z_0,z]} S(w)\,dw\)\par
\(\displaystyle=\sum_{n\geq0}\frac{a_n}{n+1}(z-z_0)^{n+1}\).
& \(\displaystyle\int_{x_0}^x S(t)\,dt\)\par
\(\displaystyle=\sum_{n\geq0}\frac{a_n}{n+1}(x-x_0)^{n+1}\).\\\hline
\textbf{Termwise differentiation}
& \(\displaystyle S'(z)=\sum_{n\geq1}na_n(z-z_0)^{n-1}\).
& \(\displaystyle S'(x)=\sum_{n\geq1}na_n(x-x_0)^{n-1}\).\\\hline
\end{longtable}
The last three rows assume \(R>0\) and apply within the radius of convergence.

\clearpage
\begin{longtable}{|>{\raggedright\arraybackslash}p{2.6cm}|*{2}{>{\raggedright\arraybackslash}p{\toolcol}|}}
\caption{Planar fields and the line-integral calculus.}\label{tab:s4-fields}\\
\hline\rowcolor{gray!7}\textbf{Property} & \textbf{Complex: \(f=u+iv\)} & \textbf{Vector: \(\vec f=(u,v)\)}\\\hline
\endfirsthead
\multicolumn{3}{l}{\textbf{Table \thetable\ (continued). Planar fields}}\\[3pt]
\hline\rowcolor{gray!7}\textbf{Property} & \textbf{Complex} & \textbf{Vector}\\\hline
\endhead
\endfoot\endlastfoot
\textbf{Limits and continuity}
& Equivalent to the respective properties of \(u,v\) at \((x_0,y_0)\), where \(z_0=x_0+iy_0\).
& Equivalent to the respective properties of \(u,v\) at \((x_0,y_0)\).\\\hline
\textbf{Integral as a limit of sums}
& \(\displaystyle\int_\gamma f(z)\,dz\)\par
\(\displaystyle=\lim_{|P|\to0}\sum_j f(\gamma(\tau_j))\Delta\gamma_j\).\par
Complex multiplication.
& \(\displaystyle\int_\gamma\vec f\cdot d\vec r\)\par
\(\displaystyle=\lim_{|P|\to0}\sum_j\vec f(\vec\gamma(\tau_j))\cdot\Delta\vec\gamma_j\).\par
Scalar product.\\\hline
\textbf{Domain of integration}
& \(\gamma(t)=x(t)+iy(t)\);\par
\(\Delta\gamma_j=\gamma(t_j)-\gamma(t_{j-1})\).
& \(\vec\gamma(t)=(x(t),y(t))\);\par
\(\Delta\vec\gamma_j=\vec\gamma(t_j)-\vec\gamma(t_{j-1})\).\\\hline
\textbf{Linearity}
& \(\displaystyle\int_\gamma(\alpha f+\beta g)\,dz\)\par
\(\displaystyle=\alpha\int_\gamma f\,dz+\beta\int_\gamma g\,dz\),\par
\(\alpha,\beta\in\mathbb C\).
& \(\displaystyle\int_\gamma(\alpha\vec f+\beta\vec g)\cdot d\vec r\)\par
\(\displaystyle=\alpha\int_\gamma\vec f\cdot d\vec r+\beta\int_\gamma\vec g\cdot d\vec r\),\par
\(\alpha,\beta\in\mathbb R\).\\\hline
\textbf{Path reversal}
& \(\displaystyle\int_{\gamma^{-1}} f\,dz=-\int_\gamma f\,dz\).
& \(\displaystyle\int_{\gamma^{-1}}\vec f\cdot d\vec r=-\int_\gamma\vec f\cdot d\vec r\).\\\hline
\textbf{Concatenation}
& \(\displaystyle\int_{\gamma_1\cdot\gamma_2}f\,dz\)\par
\(\displaystyle=\int_{\gamma_1}f\,dz+\int_{\gamma_2}f\,dz\).
& \(\displaystyle\int_{\gamma_1\cdot\gamma_2}\vec f\cdot d\vec r\)\par
\(\displaystyle=\int_{\gamma_1}\vec f\cdot d\vec r+\int_{\gamma_2}\vec f\cdot d\vec r\).\\\hline
\textbf{Change of parameter}
& Unchanged by an orientation-preserving piecewise \(C^1\) change of parameter.
& Unchanged by an orientation-preserving piecewise \(C^1\) change of parameter.\\\hline
\textbf{Integral estimate}
& \(\displaystyle\left|\int_\gamma f\,dz\right|\leq\int_\gamma|f|\,ds\leq ML\),\par
\(M=\sup_\gamma|f|\), \(L=\operatorname{length}(\gamma)\).
& \(\displaystyle\left|\int_\gamma\vec f\cdot d\vec r\right|\leq\int_\gamma\|\vec f\|\,ds\leq ML\),\par
\(M=\sup_\gamma\|\vec f\|\), \(L=\operatorname{length}(\gamma)\).\\\hline
\end{longtable}
\endgroup

On each smooth piece,
\(\Delta\gamma_j=\gamma'(\tau_j)\Delta t_j+o(\Delta t_j)\), uniformly
as the mesh tends to zero. The same statement holds for the vector
parametrization. \emph{The Riemann-sum pattern and the symbolic integral
rules agree; the products differ:}
\[
 \vec f\cdot d\vec r=u\,dx+v\,dy,\qquad
 f\,dz=(u\,dx-v\,dy)+i(v\,dx+u\,dy).
\]
These rules require no holomorphy. On a domain, vanishing of all
closed-path integrals detects a potential for a continuous vector field
or a primitive for a continuous complex function, by path integration.
Existence of a complex primitive in Table~\ref{tab:s4-scalar} is a
hypothesis, not an automatic consequence of continuity.

\Needspace{45\baselineskip}
\section{From the Core Principle to the Foundational Results of Complex Analysis}
\label{app:foundational-map}

The diagram summarizes how condition (H) generates the foundational results developed in Section~2.

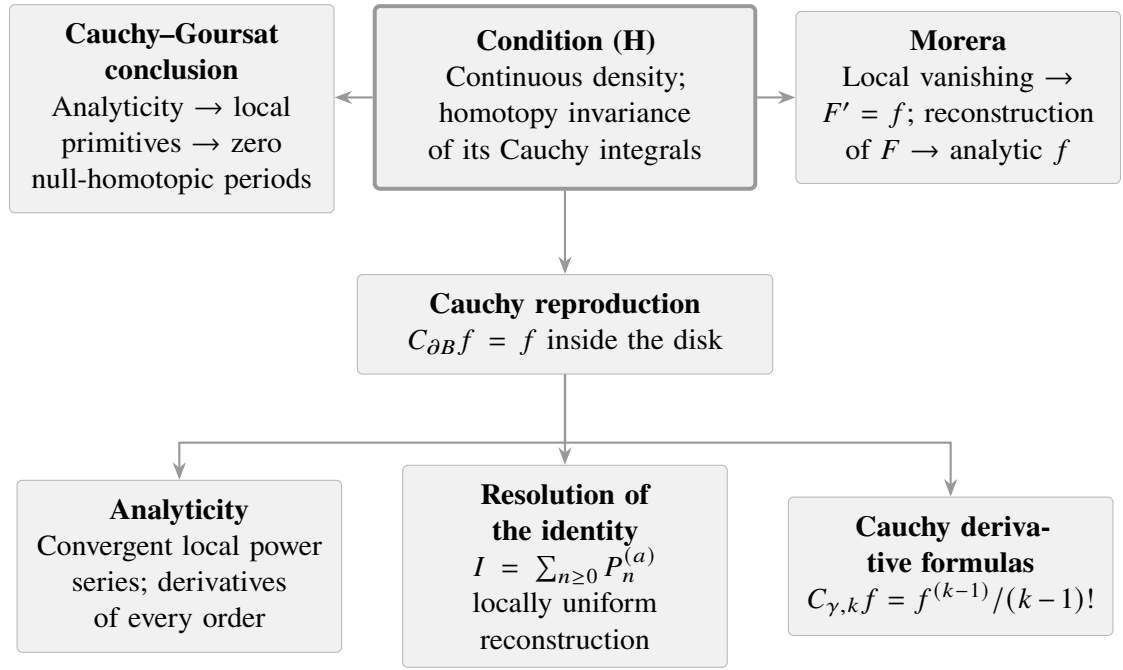
\begin{center}
\begin{tikzpicture}[>=Stealth,
 box/.style={draw=gray!55,fill=gray!11,rounded corners=2pt,
 align=center,text width=4cm,inner sep=7pt,font=\small},
 arr/.style={->,thick,gray!80}]
\node[box,text width=4.4cm,inner sep=9pt,line width=1.5pt,draw=gray!80] (H) at (0,0)
 {\textbf{Condition (H)}\\Continuous density; homotopy invariance of its Cauchy integrals};
\node[box,text width=3.8cm,anchor=north] (G) at ([xshift=-5.2cm]H.north)
 {\textbf{Cauchy--Goursat conclusion}\\Analyticity $\to$ local primitives $\to$ zero null-homotopic periods};
\node[box,text width=3.8cm] (M) at (5.2,0)
 {\textbf{Morera}\\Local vanishing $\to F'=f$; reconstruction of $F$ $\to$ analytic $f$};
\draw[arr] (H.west)--(G.east |- H.west);
\draw[arr] (H)--(M);
\node[box,text width=5cm] (C) at (0,-3)
 {\textbf{Cauchy reproduction}\\$C_{\partial B}f=f$ inside the disk};
\draw[arr] (H)--(C);
\node[box,text width=3.8cm] (A) at (-5.1,-6.2)
 {\textbf{Analyticity}\\Convergent local power series; derivatives of every order};
\node[box,text width=3.8cm] (I) at (0,-6.2)
 {\textbf{Resolution of the identity}\\$I=\sum_{n\geq0}P_n^{(a)}$\\locally uniform reconstruction};
\node[box,text width=3.8cm] (D) at (5.1,-6.2)
 {\textbf{Cauchy derivative formulas}\\$C_{\gamma,k}f=f^{(k-1)}/(k-1)!$};
\draw[arr] (C.south)--++(0,-0.9)-|(A.north);
\draw[arr] (C.south)--(I.north);
\draw[arr] (C.south)--++(0,-0.9)-|(D.north);
\end{tikzpicture}
\end{center}

\begingroup\small
\noindent\textbf{Downward arrows.} Shrinking circles and continuity give
reproduction; geometric series, estimates, termwise calculus, and kernel
differentiation give the three lower conclusions.

\smallskip\noindent\textbf{Upper branches.} Local primitives give the
Cauchy--Goursat conclusion (Lemma~\ref{lem:s2-local-primitives}). Morera
starts with its own vanishing-integral hypothesis: construct \(F'=f\),
apply reconstruction to \(F\), and differentiate
(Theorem~\ref{thm:s2-morera}). The side arrows record reuse of the mechanism.

\smallskip\noindent\textbf{Equivalent entry points.} The classical
Cauchy--Goursat route establishes (H) from complex differentiability;
a power-series-first route starts at the resolved identity and recovers
local primitives, invariance, and the integral formulas. The diagram
therefore locates the foundational meeting with both the geometric and
the power-series approaches; reverse arrows are omitted for clarity.
\endgroup

\end{document}